\documentclass[reqno, 10pt]{article}
\usepackage[latin1]{inputenc}
\usepackage{amsfonts}
\usepackage{amsmath}
\usepackage{amssymb}
\usepackage{amsthm}
\usepackage{fontenc}
\usepackage{graphicx}
\usepackage[margin=1.2in]{geometry}
\usepackage{authblk}

\newtheorem{theorem}{Theorem}[section]
\newtheorem{lemma}[theorem]{Lemma}
\newtheorem{proposition}[theorem]{Proposition}

\newtheorem{corollary}[theorem]{Corollary}
\newtheorem{definition}{Definition}

\newtheorem{question}{Question}

\numberwithin{equation}{section}

\title{Mass Partitions Via Equivariant Sections of Stiefel Bundles}

\author{\Large Steven Simon}%
\affil{\small \it Bard College\\ \it ssimon@bard.edu}
\date{}

\begin{document}

\maketitle

\begin{abstract} We consider a geometric combinatorial problem naturally associated to the geometric topology of certain spherical space forms. Given a collection of $m$ mass distributions on $\mathbb{R}^n$, the existence of $k$  affinely independent regular $q$-fans, each of which equipartitions each of the measures, can in many cases be deduced from the existence of a $\mathbb{Z}_q$-equivariant section of the Stiefel bundle $V_k(\mathbb{F}^n)$ over $S(\mathbb{F}^n)$, where $V_k(\mathbb{F}^n)$ is the Stiefel manifold of all orthonormal $k$-frames in $\mathbb{F}^n,\, \mathbb{F} = \mathbb{R}$ or $\mathbb{C}$, and $S(\mathbb{F}^n)$ is the corresponding unit sphere. For example, the parallelizability of $\mathbb{R}P^n$ when $n = 2,4$, or $8$ implies that any two masses on $\mathbb{R}^n$ can be simultaneously bisected by each of $(n-1)$ pairwise-orthogonal hyperplanes, while when $q=3$ or 4, the triviality of the circle bundle $V_2(\mathbb{C}^2)/\mathbb{Z}_q$ over the standard Lens Spaces $L^3(q)$ yields that for any mass on $\mathbb{R}^4$, there exist a pair of complex orthogonal regular $q$-fans, each of which  equipartitions the mass. \end{abstract}

\section{Introduction}

\subsection{Measure Equipartitions and Topological Combinatorics}

Topological methods are a well-established tool in combinatorial and discrete geometry, and have been proven especially powerful in the context of \textit{measure equipartition problems}. Given any collection $\mu_1,\ldots, \mu_m$ of absolutely continuous measures on $\mathbb{R}^n$ (called \textit{mass distributions}) one seeks a partition $\mathcal{P}= \{\mathcal{R}_1,\ldots, \mathcal{R}_k\}$ of $\mathbb{R}^n$ by a fixed class of interior disjoint regions so that that each region contains an equal fraction of each total measure:
					\begin{equation*} \mu_i(\mathcal{R}_j)=\frac{1}{k} \mu_i(\mathbb{R}^n) \end{equation*}  for each $1\leq j\leq k$ and each $1\leq i \leq m$. 
					
 Many different partition-types have been considered, e.g., those determined by arrangements of one or more hyperplanes [9-11, 16, 17, 22], various types of fans [5-6, 14] , polyhedral wedges and cones [25-26],  and so on. While each collection of regions posses  underlying geometric and/or permutative symmetries, the mass distributions themselves can be quite asymmetrical, so that direct solutions to such such problems have proven elusive. Instead, these symmetries can be exploited via a natural reduction of  the problem to an equivariant topological framework, the Configuration-Space/Test-Map (CS/TM)  scheme pioneered by \v Zivaljevi\'c (see, e.g., [18, 27-28]),  so that the desired equipartition is identified with a zero of an equivariant map $f: X\stackrel{G}{\rightarrow} W$ from a $G$-space $X$ with a $G$-representation $W$. The existence of the equipartition then follows by demonstrating the non-existence of any equivariant map from $X$ to the representation sphere $S(W)$, which in favorable circumstances can be precluded by vast and powerful algebraic invariants -- equivariant cohomology and ideal-valued index theory, fiber bundles and equivariant obstruction theory, spectral sequences, and so on.
	
\subsection{Our Mass Partition Problem}

We shall consider a mass partition problem where the desired equipartitions are by regular $q$-fans.

\begin{definition} A \textit{$q$-fan $F_q$} in $\mathbb{R}^n$ is the union of $q$ half-hyperplanes centered about a $(2n-2)$-dimensional affine space. A $q$-fan is called \textit{regular} if the angle between successive half-hyperplanes is always $2\pi/q$, in which case the corresponding regions $S_0,\ldots, S_{q-1}$ between the half-hyperplanes are called \textit{regular $q$-sectors}.\end{definition}  

	For example, a $q$-fan in the plane is a collection of $q$ rays emanating from a common point, and a regular $2$-fan in any dimension is a hyperplane. Given a mass distribution $\mu$ on $\mathbb{R}^n$, one says that a regular $q$-fan \textit{equipartitions} $\mu$ if $\mu(S_i)=\frac{1}{q}\mu(\mathbb{R}^n)$ for each $0\leq i<q$. We shall be concerned with finding a family of linearly independent or orthogonal $q$-fans, each of which equipartitions a given collection of mass distributions.
	
\begin{definition} Regular $q$-fans $F_q^1,\ldots, F_q^r$  in $\mathbb{R}^n$ are \textit{linearly independent} (\textit{pairwise orthogonal}) if $F_q^{1\perp}, \ldots, F_q^{r\perp}$ are linearly independent (pairwise orthogonal), where $F_q^{i\perp}$ is the linear span of the normal vectors of the half-hyperplanes of $F_q^i$.\end{definition}
		
\noindent We may now state the main concern of this paper:
		
\begin{question}  What is the maximum number $k=\Omega(q;m,n)$ (or $k=\Omega^\perp(q;m,n)$) such that for any $m$ mass distributions $\mu_1,\ldots,\mu_m$ on $\mathbb{R}^n$, there exist $k$ linearly independent (respectively, orthogonal) regular $q$-fans $F_q^1,\ldots, F_q^k$, each of which equipartitions each $\mu_j$?\end{question}

\subsection{Preliminary Observations}

	Clearly, $\Omega^\perp(q;m,n)\leq \Omega(q;m,n)$. Note that there is a natural distinction between the cases $q=2$ and $q\geq 3$. For instance, $\Omega^\perp(2;1,n)=n$ for all $n$, as one sees by considering the translated coordinate hyperplanes $H_i(t)=\{x \in\mathbb{R}^n\mid x_i=t\}$ for each $t\in\mathbb{R}$ and applying the intermediate value theorem. On the other hand, the rigidity of the angles of a regular $q$-fan implies that $\Omega(q;1,2)=0$ when $q\geq 5$, as one sees by considering two unit disks in the plane at a sufficient distance (we shall show that $\Omega(3;1,2)=\Omega(4;1,2)=1$). Secondly, the condition on linear independence of fans is different in the two cases, as $F_q^\perp$ is one-dimensional when $q=2$ and is two-dimensional otherwise. This distinction manifests itself in the upper bounds  
	
	\begin{equation} \Omega(2;m,n)\leq n-m+1 \hspace{.1in} \textit{and} \hspace{.1in} \Omega(q;m,n)\leq \lfloor\frac{n-m+1}{2}\rfloor , \hspace{.05in} q\geq 3, \end{equation} 

\noindent which can be derived by examining $m$ disjoint balls in $\mathbb{R}^n$ with affinely independent centers: If $F_q^1,\ldots, F_q^k$ are linearly independent regular $q$-fans, $q\geq 3$, each of which equipartitions each ball, then each centering codimension two affine subspace $A_i$ of $F_q^i$ must pass through the center of each of the $m$ balls, so that dim($\bigcap_{i=1}^k A_i)\geq m-1$. On the other hand, dim($\bigcap_{i=1}^k A_i)\leq n -2k$ by linear independence of these $q$-fans, yielding the second upper bound. A similar argument involving the intersection of linearly independent hyperplanes which bisect each ball establishes the first inequality.
		
\section{Main Results and Computations}
		
	We shall establish  lower bounds for $\Omega(q;m,n)$ and $\Omega^\perp(q;m,n)$ which arise from the existence of $\mathbb{Z}_q$-equivariant sections of Stiefel Bundles over spheres, or equivalently from sections of the corresponding quotient bundles over Real Projective Space when $q=2$ and over the standard Lens Spaces when $q\geq3$. Let $\rho(2; \mathbb{R}, n)$, $\rho(q; \mathbb{R}; 2n)$, and $\rho(q; \mathbb{C}, n)$ denote the maximum number $k$ for which the bundles (3.2), (3.3), and (3.4) admit sections, respectively. Then
	
\begin{theorem} $\Omega^\perp(2;m,n)\geq \min\{\rho(2;\mathbb{R},n), \frac{n-1}{m-1}\}$ \end{theorem}

\begin{theorem} For $q$ an odd prime,\\
(a) $\Omega^\perp(q;m, (q-1)n)\geq \min\{\rho(q;\mathbb{C},(q-1)n/2),\frac{n}{m}\}$ and \\
(b) $\Omega(q;m,(q-1)n)\geq \lceil \frac{1}{2}\min\{\rho(q;\mathbb{R}, (q-1)n), \frac{n}{m} \}\rceil$  \end{theorem}

\begin{theorem} (a) $\Omega^\perp(4;1,2n)\geq \min\{\rho(4;\mathbb{C},n), 2n-1\}$ and \\
(b) $\Omega(4;1,2n)\geq \lceil\frac{1}{2}\min\{\rho(4;\mathbb{R},2n), 2n-1\}\rceil$\\\end{theorem}

\subsection{Some Calculations}

	We give some estimates on $\Omega^\perp(q;m,n)$, exact in a number of cases, which follow as consequences of the above theorems. First, note that Theorem 2.1 and the estimate (1.1) show that $\Omega^\perp(2;n,n)=1$, thereby recovering the well-known Ham Sandwich Theorem - any $n$ mass distributions on $\mathbb{R}^n$ can be bisected by a single hyperplane. When $q$ is an odd prime, Theorem 2.2 yields that $\Omega(q; n, (q-1)n)\geq 1$ (see also [23]) and $\Omega^\perp(q;n, 2(q-1)n)\geq2$. In particular, one has the following corollary:
	
\begin{corollary} For $q$ an odd prime,\\  
(a) $\Omega(q;1,q-1)\geq 1$ and \\ 
(b) $\Omega^\perp(q; 1, 2q-2)\geq 2$ \end{corollary} 

	Thus any mass distribution on $\mathbb{R}^{q-1}$ can be bisected by a regular $q$-fan, while for any mass distribution on $\mathbb{R}^{2q-2}$ there exists a pair of orthogonal $q$-fans, each of which equipartitions the measure. In general, our best estimates occur for $\mathbb{R}^n$ when $n$ is a relatively small even number, in which case the quotiented Stiefel bundles admit a relatively large number of sections:

\begin{equation}\begin{matrix}
\Omega(3;1,2)=1  &  \Omega(4;1,2)=1 & \Omega^\perp(2;2,4) =3\\ 
\Omega^\perp(3;1,4) =2 & \Omega^\perp(4;1,4) =2 & \Omega(5;1,4)\geq 1\\ 
\Omega (3;2,4)=1 & \Omega (7;1,6)\geq 1 & \Omega^\perp(2;2,8)=7\\ 
\Omega^\perp(2;3,8)\geq 3 & \Omega(4;1,8)\geq 3 & \Omega^\perp(3;2,8)\geq 2\\
\Omega^\perp(5;1,8)\geq 2 & \Omega(11;1,10)\geq 1 & \Omega^\perp(7; 1,12)\geq 2\\ 
\Omega^\perp (2;2,16)\geq 9 & \Omega(3;1,16)\geq 4 & \Omega(4;1,16)\geq 4 
\end{matrix} \end{equation} 	

	In section 6, we give a detailed comparison of the above estimates with previous estimates of $\Omega(q;m,n)$ and $\Omega^\perp(q;m,n)$, as well as how they compare to the results of some related questions. The hyperplane case and its proof is discussed in section 4, while the regular $q$-fan case for odd primes $q$ is handled in section 5,  as is the case of regular 4-fans and what interesting (near) equipartition statements can be made be for regular $2q$-sectors when $q$ is an odd prime (Theorem 5.5.). First, we provide some justification for the topological nature of our lower bounds. 

\section{Topological Justification}

	The lower bounds of Theorems 2.1, 2.2, and 2.3 depend directly on the geometric topology of Real Projective Space $\mathbb{R}P^{n-1}$ and the standard Lens Spaces $L^{2n-1}(q)$. Our first indication that these manifolds are naturally involved in answering Question 1.1 arises from a canonical identification of  the space $\mathfrak{F}(2;n)$ of all hyperplanes in $\mathbb{R}^n$ with the tautological real line bundle over $\mathbb{R}P^{n-1}$, and of the space $\mathfrak{F}_{\mathbb{C}}(q; 2n)$ of all \textit{complex} regular $q$-fans in $\mathbb{R}^{2n}$ with the tautological complex line bundle over $L^{2n-1}(q)$. In other words, these regular $q$-fans are the natural combinatorial objects associated to these spherical space forms.

	To see these identifications, note that each hyperplane $H(\mathbf{a},b)=\{\mathbf{u}\in\mathbb{R}^n\mid \langle \mathbf{u}, \mathbf{a}\rangle=b\}$ in $\mathbb{R}^n$ is uniquely determined by a pair of antipodal points $\{(\mathbf{a},b), (-\mathbf{a},-b)\}$ in $S^{n-1}\times\mathbb{R}$, so that  $\mathfrak{F}(2;n)$ is realizable as the orbit space of $S^{n-1}\times\mathbb{R}$ under the diagonal $\mathbb{Z}_2$-antipodal action, i.e., with the canonical line bundle $E(\gamma)=(S^{n-1}\times\mathbb{R})/\mathbb{Z}_2$ over $\mathbb{R}P^{n-1}$.  Letting $\langle\mathbf{z},\mathbf{a}\rangle_{\mathbb{C}}=\sum_{i=1}^nz_i\bar{a}_i$ denote the standard Hermitian form on $\mathbb{C}^n$, one similarly sees that each complex regular $q$-fan in $\mathbb{R}^{2n}$ (i.e., one centered about a complex hyperplane) is of the form $F^{\mathbb{C}}_q(\mathbf{a},b)=\{\mathbf{z}\in\mathbb{C}^n\mid\langle\mathbf{z},\mathbf{a}\rangle_{\mathbb{C}}=\bar{b} + v; \arg(v)=2\pi j/q, 0\leq j <q\}$. Thus each complex regular $q$-fan can be uniquely identified with an element of  the orbit space of $S^{2n-1}\times\mathbb{C}$ under the standard diagonal $\mathbb{Z}_q$-action, i.e, with an element of the tautological complex line bundle $E(\gamma_q)=(S^{2n-1}\times\mathbb{C})/\mathbb{Z}_q$ over $L^{2n-1}(q)=S^{2n-1}/\mathbb{Z}_q$.

\subsection{Stiefel Bundles and their Quotients} Let $\mathbb{F}=\mathbb{R}$ or $\mathbb{C}$. The Stiefel manifold $V_k(\mathbb{F}^n)$ consists of all $k$-frames of $\mathbb{F}^n$, i.e., $k$-tuples $(v_1,\ldots,v_k)$ of orthonormal vectors of $\mathbb{F}^n$, and is topologized as a subset of the $k$-fold product $S(\mathbb{F}^n)\times\ldots\times S(\mathbb{F}^n)$, where $S(\mathbb{F}^n)$ is the unit sphere in $\mathbb{F}^n$ (see, e.g., [13]). In particular, $V_1(\mathbb{F}^n)= S(\mathbb{F}^n)$, and there is a fiber bundle 

\begin{equation} V_{k-1}(\mathbb{F}^{n-1})\hookrightarrow V_k(\mathbb{F}^n) \stackrel{\pi}{\rightarrow} S(\mathbb{F}^n) \end{equation}
	
\noindent given by projecting a $k$ frame $(v_1,\ldots,v_k)$ onto $\pi(v_1,\ldots, v_k) =v_1$. 

	Diagonally extending the antipodal action on $S^{n-1}$ gives a free $\mathbb{Z}_2$-action on $V_k(\mathbb{R}^n)$: $-1\cdot (v_1,\ldots, v_k) = (-v_1, \ldots, -v_k)$, and quotienting (3.1) by this action yields a fiber bundle 
	
	\begin{equation} V_{k-1}(\mathbb{R}^{n-1})/\mathbb{Z}_2\hookrightarrow V_k(\mathbb{R}^n)/\mathbb{Z}_2\stackrel{\bar{\pi}}{\rightarrow} \mathbb{R}P^{n-1} \end{equation} 

\noindent over $\mathbb{R}P^{n-1}$. Similarly, the free $\mathbb{Z}_q$-action on both $V_k(\mathbb{R}^{2n})$ and $V_k(\mathbb{C}^n)$ generated by $\zeta_q\cdot (v_1,v_2, \ldots, v_k) = (\zeta_qv_1, \zeta_q^rv_2, \ldots, \zeta_q^rv_k)$, $\zeta_q=e^{\frac{2\pi i q}{m}}$ and $1\leq r<q$ relatively prime to $q$, yields fiber bundles  over $L^{2n-1}(q)$:
	
	 \begin{equation} V_{k-1}(\mathbb{R}^{2n-1})/\mathbb{Z}_q \hookrightarrow V_k(\mathbb{R}^{2n})/\mathbb{Z}_q \stackrel{\bar{\pi}}{\rightarrow} L^{2n-1}(q) \end{equation}
	 
	 \begin{equation} V_{k-1}(\mathbb{C}^{n-1})/\mathbb{Z}_q \hookrightarrow V_k(\mathbb{C}^n)/\mathbb{Z}_q \stackrel{\bar{\pi}}{\rightarrow} L^{2n-1}(q) \end{equation} 
	 
We shall want to find a $\mathbb{Z}_q$-equivariant section of the bundle (3.1), or in other words a section of the corresponding quotient bundle (3.2), (3.3), and (3.4):  a continuous map $s: S(\mathbb{F}^n)\rightarrow V_k(\mathbb{F}^n)$ such that ($i$) $\pi\circ s(v)= v$ for each $v\in S(\mathbb{F}^n)$ (i.e., each $(v,s(v))$ is a $k$-frame in $\mathbb{F}^n$) and ($ii$) $(\zeta_q v, s(\zeta_q v)) = \zeta_q \cdot (v, s(v))$.  

	As a section of (3.2) is the same as the existence of $(k-1)$ orthonormal vector fields on $\mathbb{R}P^{n-1}$, we have $\rho(2;\mathbb{R},n)=Span(\mathbb{R}P^{n-1})+1$, where we recall that $Span(M^n)$ denotes the maximum number of linearly independent (equivalently, orthonormal) vector fields on a smooth manifold $M^n$. Likewise, a section of (3.3) is the same as the existence of $(k-1)$ orthonormal vector fields on $L^{2n-1}(q)$ when the $\mathbb{Z}_q$-action above on $V_k(\mathbb{R}^{2n})$ is standard (i.e., $r=1$), and in fact $\rho(q;\mathbb{R};2n) =$ Span$(L^{2n-1}(q))+1$ by the explicit calculations of Becker [7] and Szczarba [24] (see, e.g., [15] as well).  
	
	Owing to the identification of regular (complex) $q$-fans and the aforementioned tautological bundles, the connection between Question 1.1 and the tangent bundles of these manifolds can be therefore be glimpsed philosophically in terms of the fundamental relationship (see, e.g., [19, 24]) between the tangent bundles over these manifolds and the $n$-fold Whitney sum of their tautological bundles:
	
\begin{equation} T\mathbb{R}P^{n-1}\oplus \varepsilon^1 = E(\gamma)\oplus\ldots\oplus E(\gamma) \end{equation} and 
\begin{equation} TL^{2n-1}(q) \oplus \varepsilon^1 = Re(E(\gamma_q))\oplus\ldots\oplus Re(E(\gamma_q)), \end{equation} 

\noindent where $Re(E(\gamma_q)$ denotes the underlying 2-dimensional real bundle of $E(\gamma_q)$ and $\varepsilon^1$ is the trivial line bundle obtained by quotienting the normal bundle of the sphere. A similar analysis holds in the complex case by quotienting the complex tangent bundle $T_{\mathbb{C}}S^{2n-1}=\{(v,w)\mid \langle v, w\rangle_{\mathbb{C}}=0\} \subseteq S^{2n-1}\times\mathbb{C}^n$ by the $\mathbb{Z}_q$-action considered.\\

	Our proofs of the Theorems 2.1, 2.2., and 2.3  follow a similar pattern. A section of the relevant quotiented Stiefel bundle gives rise to a parametrized family of $k$-tuples of regular $q$-fans. Given an appropriate number of mass distributions, the algebraic topology of the related manifold, as realized by an appropriate equivariance statement, implies that each fan in at least one these collections must equipartition each of the given measures. In the case of hyperplanes, the orthonormality of the vector fields implies that these hyperplanes are orthogonal. For $q\geq 3$, the regular $q$-fans are orthogonal if one considers complex Stiefel manifolds, while one can ensure that half of the regular $q$-fans are linearly independent if one considers real Stiefel manifolds.
	
	It is worthwhile to note that our methods for establishing our lower bounds are in some sense the opposite of the usual topological constructions utilized in equipartition theory. In the configuration space/test-map procedure, one deduces the existence of a desired partition from the {\it non-existence} of an (equivariant) section of an appropriate bundle, while here the existence of such sections is used to establish our equipartition results. 
	
\section{Bisections by Orthogonal Hyperplanes}

As noted above, Theorem 2.1 recovers the Ham Sandwich Theorem when $m=n$, and in fact our proof of Theorem 2.1 reduces to the classical geometric proof of the Ham Sandwich Theorem as attributed to Banach by Steinhaus [8] in this case. For $m<n$, our best estimates occur when $Span(\mathbb{R}P^{n-1})$ is largest. It is a well-known fact that $\mathbb{R}P^{n-1}$ is parallelizable iff $n=2,4,8$, arising from multiplication in the division algebras of complex numbers, Quaternions, and Octonions. Thus one obtains the optimal values $\Omega^\perp(2;2,n)=n$ in these cases. For general $n=2^{a+4b}m$, $m$ odd and $0\leq a\leq 3$, the exact value of Span$(\mathbb{R}P^{n-1})= 2^a +8b$ was determined by Adams [1],  with vector fields arising from realizations of $\mathbb{R}^n$ as a real Clifford module (see, e.g., [3, 15]).  Thus  $\Omega^\perp(2;3,8)\geq 3$ because Span($\mathbb{R}P^7)=7$, while  $\Omega^\perp(2;2,16)\geq 9$ and $\Omega^\perp(2;3,16)\geq7$ since $Span(\mathbb{R}P^{15})=9$. Owing to the paucity of orthonormal vector fields on $\mathbb{R}P^{n-1}$,  however, our estimates on $\Omega^\perp(2;m,n)$ obtained in this fashion worsen as $n$ increases. In particular, one can compare our estimate $\Omega^\perp(2;2,n)\geq$ $\min\{$Span$(\mathbb{R}P^{n-1})+1, n-1\}$ with the known value $\Omega^\perp(2;2,n)=n-1$ given in [15]. 

\subsection{Proof of Theorem 2.1}

	We follow the program laid out in section 3. First, we have the following lemma. 
	
\begin{lemma} Let $\mu$ be a mass distribution on $\mathbb{R}^n$, and let $s: S^{n-1}\rightarrow S^{n-1}$ be a continuous $\mathbb{Z}_2$-equivariant map. Then there exists a continuous family $\{H(x)=H(-x)\}_{x\in S^{n-1}}$ of hyperplanes parametrized by $\mathbb{R}P^{n-1}$, each of which bisects $\mu$. Moreover, $H^\pm(-x)=H^\mp(x)$ for each pair of corresponding half-spaces.\end{lemma}

\begin{proof} Define a family of hyperplanes $H(x,t)=\{u\in\mathbb{R}^n\mid \langle u, s(x) \rangle = t\}$ and half-spaces $H^+(x,t)=\{\langle u, s(x)\rangle \geq t\}$ and $H^-(x,t)=\{\langle u, s(x)\rangle \leq t\}$  for each $(x,t) \in S^{n-1}\times\mathbb{R}$. As $\mu$ is a mass distribution, the association  $(x,t)\mapsto \mu(H^+(x,t))$ is continuous, while $H^\pm(-x,-t) = H^-(x,t)$ for each $(x,t)$ because $s(-x)=-s(x)$ by assumption. 
 
 	For each $x\in S^{n-1}$, define $g_x: \mathbb{R}\rightarrow [0,\infty)$ by $g_x(t) = \mu(H^+(x,t))$. The monotonicity properties of $\mu$ show that $g_x$ is monotone decreasing, with $\lim_{t\to-\infty}g_x(t) = \mu(\mathbb{R}^n)$ and $\lim_{t\to+\infty} g_x(t) = \mu(\emptyset) = 0$. Hence $g_x^{-1}(\frac{1}{2}\mu(\mathbb{R}^n))$ is a closed interval by the intermediate value theorem and the monotonicity of $g_x$, and letting $t(x)$ be the center of this interval defines a  continuous function $t:S^{n-1}\rightarrow \mathbb{R}$. As $t(-x)=-t(x)$, defining $H(x):=H(x,t(x))$ and $H^\pm(x):=H^\pm(x,t(x))$ gives the desired family of parametrized hyperplanes and half-spaces. \end{proof}

	We now prove Theorem 2.1. Let $\mu_1,\ldots, \mu_m$ be a collection of mass distributions on $\mathbb{R}^n$. For $k=\min\{\rho(2;\mathbb{R},n), \frac{n-1}{m-1}\}$, let $s(x)= (s_1(x), \ldots, s_k(x))$ be a $\mathbb{Z}_2$-equivariant map section of $V_k(\mathbb{R}^n)$, where $s_1(x)=x$. Letting $H_i(x)$ denote the hyperplane corresponding to $s_i(x)$ which bisects  $\mu_1$ as given by (the proof of) Lemma 4.1, we see that $H_1(x),\ldots, H_k(x)$ are pairwise orthogonal by the orthonormality of the $s_i(x)$. Defining $f=(f_2,\ldots,f_k): S^{n-1} \rightarrow \mathbb{R}^{k(m-1)} \subseteq \mathbb{R}^{n-1}$ by $f_j(x) = (\mu_j(H_1^+(x)), \ldots, \mu_j(H_k^+(x)))$ for each $2\leq j\leq m$, applying the Borsuk-Ulam Theorem below shows there exists some pair $\{x,-x\}\subseteq S^{n-1}$ for which $\mu_j(H_i^-(x)) = \mu_j(H_i^+(-x)) = \mu_j(H_i^+(x))$ for each $2\leq j \leq m$ and each $1\leq i \leq k$, thereby finishing the proof.
	
\begin{theorem} The Borsuk-Ulam Theorem\\ 
Let $f: S^{n-1}\rightarrow \mathbb{R}^{n-1}$ be a continuous map. Then $f(-x)=f(x)$ for some pair $\{\pm x\}\subseteq S^{n-1}$. \end{theorem} 
	
\section{Regular $q$-Fans, $q\geq 3$}

\subsection{Proof of Theorem 2.2}

\begin{proof} Let $\mu_1,\ldots, \mu_m$ be mass distributions on $\mathbb{R}^{2t}$, $t=\frac{(q-1)n}{2}$. For $k=\min\{\rho(q;\mathbb{C},t),n/m\}$ (respectively, $k=\min\{\rho(q;\mathbb{R}, 2t),n/m\}$), let $s(x)=(s_1(x)=x,s_2(x) \ldots, s_k(x))$ be a  section of $V_k(\mathbb{C}^t)$ (respectively, $V_k(\mathbb{R}^{2t})$), equivariant with respect to a given $\mathbb{Z}_q$-action on $V_k(\mathbb{C}^{t})$ (respectively, $V_k(\mathbb{R}^{2t})$), and extend each $s_i$ to a $\mathbb{Z}_q$-equivariant map $\tilde{s}_i:\mathbb{R}^{2t}\rightarrow \mathbb{R}^{2t}$ by setting $\tilde{s}_i(ax)=as_i(x)$ for each $a\geq 0$ and each $x\in S^{2t-1}$.

	Fix $1\leq i \leq k$. For each $\mathbf{u}=(u_0,\ldots, u_t)\in S^{2t +1}\subseteq\mathbb{C}^{t+1}$, we define $S_j^i(\mathbf{u}) :=  \{\mathbf{z}\in \mathbb{C}^t\mid \langle \mathbf{z}, \tilde{s}_i(u_1,\ldots, u_t) \rangle_{\mathbb{C}} = - (\bar{u}_0)^r +v; \arg(v)\in [\frac{2\pi j}{q}, \frac{2\pi (j+1)}{q}]\}$ for each $0\leq j<q$. For the given $1\leq r<q$, $\tilde{s}_i(\zeta_q u)=\zeta_q^r\tilde{s}_i(u)$ for each $u\in\mathbb{R}^{2t}$, and so $S^i_0(\zeta_q^j \mathbf{u})=S^i_{jr}(\mathbf{u})$ (mod $q$) for each $0 \leq j <q$.
	
	To ensure continuity of our construction, we exclude from $S^{2t +1}$ the set $X_q:=\{\zeta_q^j\}_{j=0}^{q-1}\times 0$. When $\mathbf{u}\notin S^1\times 0$, the $S_j^i(\mathbf{u})$ are the regular $q$-sectors of the complex regular $q$-fan $F^{\mathbb{C},i}_q(\mathbf{u}):=F^{\mathbb{C}}_q(s_i(\frac{(u_1,\ldots, u_t)}{||(u_1,\ldots, u_t)||}), \frac{-u_0^r}{||(u_1,\ldots, u_t)||}))$. On the other hand, if $\mathbf{u}=(u_0,0)\in S^1 \times 0$, that $u_0\neq \zeta_q^j$ for any $0\leq j<q$ implies that for each $1\leq i\leq k$ there exists some unique $0\leq j_i<q$ such that $S^i_{j_i}(\mathbf{u}) = \mathbb{R}^{2t}$  and $S^i_j(\mathbf{u})=\emptyset$ for $j \neq j_i$.
	
	Now let  $\mu_1,\ldots, \mu_m$ be mass distributions on $\mathbb{R}^{2t}$. For each $1\leq \ell \leq m$, let $f^\ell=(f^\ell_1,\ldots, f^\ell_k): S^{2t+1} - X_q \rightarrow \mathbb{R}^k$ be given by  $f^\ell_i(\mathbf{u}) = \mu_\ell(S_0^i(\mathbf{u}))$ for each $1\leq i \leq k$, and let $f=(f^1,\ldots, f^m): S^{2t+1}- X_q\rightarrow \mathbb{R}^{km}\subseteq\mathbb{R}^n$. It follows as in [23] that $f$  is continuous, so by Theorem 5.2 below there exists some $\mathbb{Z}_q$-orbit $\{\zeta_q^j\mathbf{u}\}_{j=0}^{q-1} \subseteq S^{2t +1} - X_q$ for which $\mu_\ell(S^i_{jr}(\mathbf{u})) = \mu_\ell(S^i_0(\zeta_q^j\mathbf{u}))=\mu_\ell(S^i_0(\mathbf{u}))$ for each $1\leq i \leq k$, $0\leq j < q$, and $1\leq \ell \leq k$. By the discussion above, this $\mathbf{u}$ cannot be in $S^1\times 0$, lest $\mu_\ell(\mathbb{R}^{2t})=\mu(\emptyset)=0$. Thus each $F^{\mathbb{C},i}_q(\mathbf{u})$ is a complex regular $q$-fan, and as $\sum_{j=0}^{q-1} \mu_\ell(S^i_{jr}(\mathbf{u}))=\mu_\ell(\mathbb{R}^{2t})$, each $F^{\mathbb{C},i}_q(\mathbf{u})$ equipartitions each $\mu_\ell$. 
	
	Supposing $s$ is a section of $V_k(\mathbb{C}^t)$, the $s_i(\frac{(u_1,\ldots, u_t)}{||(u_1,\ldots, u_t)||})$ are complex orthonormal, and hence the centering complex hyperplanes  of the $F^{\mathbb{C},i}_q(\mathbf{u})$ are pairwise complex orthogonal. Thus these regular $q$-fans are pairwise orthogonal. In the case $s$ is a section of $V_k(\mathbb{R}^{2t})$, Proposition 5.1 applied to the orthonormal vectors $s_1(\frac{(u_1,\ldots, u_t)}{||(u_1,\ldots, u_t)||}),\ldots, s_k(\frac{(u_1,\ldots, u_t)}{||(u_1,\ldots, u_t)})$ shows that  at least $\lceil \frac{k}{2}\rceil$ of these vectors are complex linearly independent, and hence that at least $\lceil \frac{k}{2}\rceil$ of the regular $q$-fans are linearly independent.\end{proof}	

\begin{proposition} If $u_1,\ldots, u_k$ are linearly independent vectors in $\mathbb{R}^{2t}$, then at least $r=\lceil \frac{k}{2}\rceil$ of the $u_j$ are complex linearly independent in $\mathbb{C}^t$.\end{proposition}

\begin{proof} We argue by induction of $k$. For $k=1$ there is nothing to show, and for the induction step it is enough to assume that $k=2r$. Suppose that $u_1,\ldots, u_{k+1}$ are linearly independent in $\mathbb{R}^{2t}$. By induction, we may assume that the vectors $u_1,\ldots, u_r$ are complex linearly independent in $\mathbb{C}^t$, i.e., that $u_1,iu_1,\ldots, u_r, iu_r$ are linearly independent in $\mathbb{R}^{2t}$. If $u_1,\ldots, u_r, u_j$ were not complex linearly independent in $\mathbb{C}^t$ for any $r<j\leq k+1$, then the linear span of $\{u_1,iu_1,\ldots, u_r,iu_r\}$ would equal the linear span of $\{u_1,u_2,\ldots, u_{k+1}\}$, contradicting the linear independence of these vectors. \end{proof}
	
\begin{theorem} Let $q$ be an odd prime. Given any continuous map $f: S^{(q-1)n+1}-X_q\rightarrow \mathbb{R}^n$, there exists some $\mathbb{Z}_q$-orbit $\{\zeta_q^kx\}_{k=0}^{q-1}\subseteq S^{(q-1)n+1}-X_q$ such that $f(x)=f(\zeta_qx)=\ldots=f(\zeta_q^{q-1}x)$. \end{theorem}
	
\begin{proof}  This is a corollary of the ``$\mathbb{Z}_q$-Borsuk Ulam Theorem" considered in [23], which states that given any continuous map $g=(g_1,\ldots, g_n):S^{2n+1}-X_q\rightarrow \mathbb{C}^n$ and any $n$ integers $1\leq r_1,\ldots, r_n <q$ relatively prime to $q$ , there exists some $x\in S^{2n+1}-X_q$ such that $\sum_{k=0}^{q-1}\zeta_q^{-kr_i}g_i(\zeta_q^kx)=0$ for each $1\leq i \leq n$. 

	For $q$ an odd prime and $f=(f_1,\ldots, f_n):S^{(q-1)n+1}-X_q\rightarrow \mathbb{R}^n$ a continuous map, we apply this theorem to the map $g=(g_{1,1},\ldots, g_{n,(q-1)/2}):S^{(q-1)n+1}-X_q\rightarrow \mathbb{R}^{n(q-1)/2}\subseteq\mathbb{C}^{n(q-1)/2}$ and numbers $r_{1,1},\ldots, r_{n,(q-1)/2}$ by letting $g_{i,k}=f_i$ and $r_{i,k}=k$ for each $1\leq i \leq n$ and $1\leq k \leq (q-1)/2$.

Fixing $1\leq i \leq n$, we have \begin{equation} \sum_{k=0}^{q-1}\zeta_q^{-kr}f_i(\zeta_q^kx)=0\end{equation} for each $1\leq r\leq (q-1)/2$. As $\sum_{r=1}^{\frac{q-1}{2}}\cos(\frac{2\pi rk}{q})=-\frac{1}{2}$ for each $1\leq k<q$, summing (5.1) over $r$ and evaluating the real part of the ensuing equation gives $\frac{q-1}{2}f_i(x)-\frac{1}{2}(B_i(x)-f_i(x))=0$, and hence $f_i(x)=\frac{1}{q}B_i(x)$, where $B_i(x)=\sum_{k=0}^{q-1}f_i(\zeta_q^kx)$. If we multiply (5.1) by $\zeta_q^{jr}$ for each $r$, similar reasoning yields $f_i(\zeta_q^jx)=\frac{1}{q}B_i(x)$ for each $j>0$ as well. \end{proof}

\subsection{Proof of Theorem 2.3}

As in the proof of Theorem 2.1, one has the following observation concerning $\mathbb{Z}_4$-equivariant maps and regular 4-fans:

\begin{lemma} Let $\mu$ be a mass distribution on $\mathbb{R}^{2n}$. A $\mathbb{Z}_{4}$-equivariant map $s: S^{2n-1}\rightarrow S^{2n-1}$ induces a continuous family $\{F_4(x)=F_4(ix)=F_4(-x)=F_4(-ix)\}_{x\in S^{2n-1}}$ of complex regular $4$-fans parametrized by $L^{2n-1}(4)$, each of whose opposite regular $4$-sectors $S_k(x)$ and $S_{k+2}(x)=S_k(-x)$  have equal measure. \end{lemma}
	
\begin{proof} The proof is analogous to that of Lemma 4.1. For $0\leq k<4$, each pair $(x,t)\in S^{2n-1}\times\mathbb{R}$ defines a family of hyperplanes  
 $H_k(x,t)=\{u\in\mathbb{R}^{2n}\mid \langle u, i^ks(x) \rangle = t\}$ in $\mathbb{R}^{2n}$. For simplicity, assume that $s(ix)=is(x)$ for each $x\in S^{2n-1}$ (the proof when $s(ix)=-is(x)$ is similar), so that $H_k^\pm(x,t)=H_0^\pm(i^kx,t)$ and  $H_{k+2}^\pm(x,t)=H_0^\pm(-i^k x,t)=H_0^\mp(i^k x,-t)=H_k^\mp(x,-t)$.

	As before, for each $x\in S^{2n-1}$ we can choose $t_k(x)\in\mathbb{R}$ for which the hyperplane $H_k(x):=H_k(x,t_k(x))$ bisects $\mu$ and so that the association $x\mapsto t_k(x)$ is continuous. As $t_k(x)=t_0(i^kx)$ and $t_{k+2}(x)=t_k(-x)=-t_k(x)$, we have $H_k^\pm(x)=H_0^\pm(i^kx)$ and $H_{k+2}^\pm(x)=H_k^\pm(-x)=H_k^\mp(x)$. Since $s(x)$ and $is(x)$ are orthonormal, the hyperplanes $H_0(x)=H_0(-x)=H_2(x)$ and $H_1(x)=H_1(-x)=H_3(x)$ are orthogonal, and hence $F_4(x):=H_0(x)\cup H_1(x)$ is a regular 4-fan, and indeed a complex regular 4-fan because $H_0(x)\cap H_1(x)=\{u\in\mathbb{C}^n\mid \langle u, x\rangle_{\mathbb{C}}=t_0(x) +it_1(x)\}$ is a complex hyperplane. The regular 4-sectors corresponding to $F_4(x)$ are given by $S_k(x):=H_k^-(x)\cap H_{k+1}^+(x)$, so $S_{k+2}(x)=S_k(-x)$. That both $H_k(x)$ and $H_{k+1}(x)$ bisect $\mu$ is equivalent to $\mu(S_k(x))=\mu(S_{k+2}(x))=\mu(S_k(-x))$. \end{proof}

\noindent We may now prove Theorem 2.3:

\begin{proof} Let $s(x) = (s_1(x)=x, s_2,\ldots, s_k)$ be a $\mathbb{Z}_4$-equivariant section of $V_k(\mathbb{C}^n)$ (respectively, $V_k(\mathbb{R}^{2n}$)), where $k=\min\{\rho(4;\mathbb{C},n), 2n-1\}$ (respectively, $k=\min\{\rho(4;\mathbb{R},2n), 2n-1\}$). For a mass distribution $\mu$ on $\mathbb{R}^{2n}$, let $\{F_4^j(x)\}_{x\in S^{2n-1}}$ and $\{S_\ell ^j(x)\mid 0\leq \ell<4\}_{x\in S^{2n-1}}$ denote the family of complex regular 4-fans and regular 4-sectors associated to the $s_j$ as given in the proof of Lemma 5.3.

As $S^j_0(\pm ix)=S^j_1(x)$ and $S^j_1(\pm ix) = S^j_0(x)$, defining $f=(f_1,\ldots, f_k):S^{2n-1}\rightarrow \mathbb{R}^k\subseteq\mathbb{R}^{2n-1}$ by $f_j(x)=\mu(S^j_0(x))-\mu(S^j_1(x))$ for each $1\leq j\leq k$ yields $f(ix)=-f(x)$, and hence Theorem 5.4 below guarantees some $x\in S^{2n-1}$ such that $\mu(S^j_0(x))=\mu(S^j_1(x))$ for each $1\leq j \leq k$. However, $\mu(S^j_2(x))=\mu(S^j_0(x))$ and $\mu(S^j_1(x))=\mu(S^j_3(x))$, so $\mu(S^j_k(x))=\frac{1}{4}\mu(\mathbb{R}^{2n})$ for each $0\leq k<4$.  

	In the case of complex Stiefel bundles, the complex orthonormality of  $s_1(x),$ $\ldots, s_k(x)$ shows that $F_4^1(x),\ldots, F_4^k(x)$ are orthogonal, while in the case of real Stiefel bundles Proposition 5.1 applied to these vectors implies that at least $\lceil \frac{k}{2}\rceil$ of the equipartitioning regular 4-fans $F_4^1(x),\ldots, F_4^k(x)$ are linearly independent. \end{proof}

\begin{theorem} If $f: S^{2n-1}\rightarrow\mathbb{R}^{2n-1}$ is a continuous map satisfying $f(ix)=-f(x)$ for each $x\in S^{2n-1}$, then $f(x)=0$ for some $x\in S^{2n-1}$. \end{theorem}

\begin{proof} This follows from the intermediate value theorem when $n=1$ and is a standard if degree argument $n>1$: Supposing no such $x$ exists, the map $g(x)=f(x)/||f(x)||$ is continuous, as is the composition $h=j\circ g: S^{2n-1}\rightarrow S^{2n-1}$, where $j: S^{2n-2}\hookrightarrow S^{2n-1}$ is the inclusion $x\mapsto(x,0)$. As $h$ is nullhomotpic, $deg(h)=0$, but $h(ix)=-h(x)$ for each $x\in S^{2n-1}$ implies that $deg(h)\equiv 2$ \textit{mod} 4, as follows from examining the induced map $\bar{h}: L^{2n-1}(4)\rightarrow \mathbb{R}P^{2n-1}$ and the fundamental group and $\mathbb{Z}_2$-cohomology rings of the spaces involved (see, e.g., [13]). \end{proof}

\subsection{Some Lower bounds for Regular $q$-fans} 

	The values of $\rho(q;\mathbb{R},2n)$ $=Span(L^{2n-1}(q))+1$ given in [7] when $n\neq 4$ and deduced as in [14] when $n=4$ determine our lower bounds for $\Omega(q;k,2n)$. For instance, Span$(L^7(4))=5$ implies $3\leq \Omega(4;1,8)\leq 4$, while Span$(L^{16}(3))=9$ and Span$(L^{16}(4))=8$ yield the estimates $\Omega(3;1,16)\geq 4$ and $\Omega(4;16,4)\geq 4$, respectively. As with the case of hyperplanes, however, our estimates on $\Omega(q; (q-1)n)$ are relatively weak in comparison to $n$, owing to the relative lack of linearly independent vector fields on $L^{q(n-1)-1}$. For instance, one can compare Theorem 2.3 with the known value $\Omega^\perp(4;1,n)=\lfloor\frac{n-1}{2}\rfloor$ given in [16]. 

	On the other hand, the numbers $\rho(q;\mathbb{C},n)$ apparently have yet to be determined (see, e.g., [15, 20-21]), though the maximum value $k=\rho(\mathbb{C}, n)$ for which $V_k(\mathbb{C}^n)$ admits a (non-equivariant) section is known for all $n$. Unfortunately, $\rho(\mathbb{C}, n)$ (and hence $\rho(q; \mathbb{C},n)\leq \rho(\mathbb{C}, n))$ is quite small. Indeed,  $\rho(\mathbb{C},n)=1$ if $n$ is odd and $\rho(\mathbb{C}, 2k)=2$ for $1\leq k\leq 11$ (see [2, 4] for an explicit formula for arbitrary $n$). In fact, the {\it only} (see, e.g., [15, 20, 21]) explicitly known example of a section $s: S^{4n-1}\rightarrow V_2(\mathbb{C}^{2n})$ is $s(x) = (x,jx)$, where $j$ is the usual unit quaternion and $S^{4n-1}\subseteq\mathbb{H}^n$. This map is $S^1$-equivariant, with $s(\lambda x)=(\lambda x,\lambda^{-1}jx)$ for each $\lambda\in S^1$ and $x\in S^{4n-1}$, and hence is $\mathbb{Z}_q$-equivariant for each $q$. In particular, the triviality of the corresponding circle bundle $S^1\hookrightarrow V_2(\mathbb{C}^2)/\mathbb{Z}_q \rightarrow L^3(q)$ yields the optimal values $\Omega^\perp(3;1,4)=\Omega^\perp(4;1,4)=2$, while $\rho(q; \mathbb{C};4)=2$ yields the estimates $\Omega^\perp(5;1,8)\geq2$ and $\Omega^\perp(3;2,8)\geq 2$.  

\subsection{Regular $2q$-sectors} 

	Let $q$ be an odd prime. Given a mass distribution on $\mathbb{R}^{2n}$, the same reasoning as in Lemma 5.3 shows that the identity map on $S^{2n-1}$ induces a continuous family $\{H_0(x),H_1(x)=H_0(\pm \zeta_{2q}x), \ldots, H_{q-1}(x)=H_0(\pm \zeta_{2q}^{q-1}x)\}_{x\in S^{2n-1}}$ of $q$ hyperplanes in $\mathbb{R}^{2n}$ parametrized by $L^{2n-1}(2q)$, for which the angle $\langle H_k(x), H_{k+1}(x)\rangle$ between any two successive hyperplanes is always equal to $\pi/q$, and for which each pair $S_k(x)$ and $S_k(-x)=S_{k+q}(x)$ of opposite regular $2q$-sectors $S_k(x)=H_k^-(x)\cap H_{k+1}^+(x)$ have equal measure. Moreover, $S_k(x)=S_0(\zeta_{2q}^kx)$ for each $0\leq k<2q$.
		
	Theorem 5.2 still holds (with nearly identical proof) if the set $X_q$ is not removed, so in particular for any continuous map $f:S^q \rightarrow \mathbb{R}$ there exists some $x\in S^q$ for which $f(x)=f(\zeta_q^kx)$ for each $1\leq k<q$. Letting $f(x)=\mu(S_0(x))$ and noting that $\mu(S_j(x))=\mu(S_{j+q}(x))$ for all $0\leq j<q$ therefore yields the following near equipartition statement: 
	
\begin{theorem} Let $q$ be an odd prime and let $\mu$ be a mass distribution on $\mathbb{R}^{q+1}$. There exists a collection $\{H_i\}_{i=0}^{q-1}$ of $q$ hyperplanes with consecutive angles $\langle H_j, H_{j+1}\rangle =\frac{\pi}{q}$ for each $0\leq j<q$, so that $\mathbb{R}^{q+1}$ is covered by the regular $2q$-sectors $\{S_k=H_k^-\cap H_{k+1}^+\}_{k=0}^{2q-1}$ and $\mu(S_0)=\mu(S_1)=\ldots = \mu(S_{2q-1})$.\end{theorem}
	
	Such a collection of $q$ hyperplanes as given by the theorem might not constitute a regular $2q$-fan, however, since the interiors of the $2q$-sectors may intersect non-trivially. Thus for any $\mu$ on $\mathbb{R}^4$ we can say there exists hyperplanes $H_0,H_1,H_2$ whose successive angles are always equal $\pi/3$ and whose corresponding regular 6-sectors cover $\mathbb{R}^4$ and all have measure equal to $\frac{1}{6}\mu(\mathbb{R}^4)+\frac{1}{3}\mu(I)$, where $I$ is the intersection of the regular sectors. This can be compared to the result that any mass distribution on $\mathbb{R}^4$ can be equipartitioned ``\textit{modulo} 2" by a regular 6-fan [23]. A similar decomposition gives $\mathbb{R}^6$ as the union of 10 regular 10-sectors of equal measure exists for any mass distribution on $\mathbb{R}^6$, and likewise for $\mathbb{R}^8$ and regular 14-sectors. For a single measure on $\mathbb{R}^{2q+2}$, the section $s(x)=(x,jx)$ of $V_2(\mathbb{C}^{2q+2})$ can be used as before to find a pair of collections of regular $2q$-sectors covering $\mathbb{R}^{2q+2}$, all of which have equal measure, and which are orthogonal in an appropriate sense.  
	
\section{Related Questions and Results}

	We conclude this paper by discussing some research related to Question 1.1. Of most importance is an equipartition problem originally posed by Gr\"unbaum [12]. Instead of asking for a collection of $k$ hyperplanes, each of which bisects a given collection of mass distributions, one demands instead that the hyperplanes {\it equipartition} each of the measures, i.e., that the each of the $2^k$ orthants determined by the $k$ hyperplanes contains $1/2^k$  the measure of each mass distribution. A triple $(k,m,n)$ is said to be {\it admissible} if any $m$ mass distributions on $\mathbb{R}^n$ can be equipartitioned by $k$ hyperplanes, and one searches for the minimum $n=\Delta(k,m)$ for which $(k,m,n)$ is admissible.  By considering the unit ball in $\mathbb{R}^n$, one can deduce as in [21] that the admissibility of $(k,m,n)$ implies that any $(m-1)$ distributions on $\mathbb{R}^n$ can be equipartitioned by $k$ orthogonal hyperplanes. In particular, one has $\Omega^\perp(2;m-1, \Delta(k,m))\geq k$ and $\Omega^\perp(4;m-1, \Delta(2k,m))\geq k$.  
For instance, the values $\Delta(2, 2^{r+1}-1) =  3\cdot 2^r-1$ of  [16] imply that $\Omega^\perp(4; 2^{r+1}-2, 3\cdot 2^r-1)\geq 1$.
	
	A major breakthrough on estimates of $\Delta(k,m)$ was given by Ramos [22], whose results were extended substantially in [17] and then in [9-10]. Better conditions for ensuring the orthogonality of equipartitioning hyperplanes were provided in [11].  Closely related to these results are those of Makeev [16], who showed ($i$) that for a single measure on $\mathbb{R}^n$, one always find $n-1$ pairwise orthogonal hyperplanes, any two of which equipartition the measure, (thereby establishing $\Omega^\perp(4;1,n) = \lfloor\frac{n}{2}\rfloor$), and ($ii$) that for any two such measures on $\mathbb{R}^n$ one can find $(n-1)$ pairwise orthogonal hyperplanes, any two of which equipartition both measures (thereby establishing the values $\Omega^\perp(4;2,n) = \lfloor \frac{n-1}{2}\rfloor$ and $\Omega^\perp(2;2,n)=n-1$).  Finally, we note that the values $\Omega(3;1,2)=\Omega(4;1,2)=1$ were also obtained by B\'ar\'any and Matou\v sek ([5, 6]), and that $\Omega(3;n-1,n)=1$ was established in [26]. 

\section{Acknowledgment}

The author thanks his thesis advisor, Sylvain E. Cappell, whose guidance and suggestions were of great help in the development of this paper, much of which is contained in the author's doctoral thesis. This research was supported in part by DARPA grant HR0011-08-1-0092.

\bibliographystyle{plain}

\begin{thebibliography}{10}

\smallskip

\bibitem{JFA62} J.F. Adams, Vector Fields on Spheres, Ann. Math. 75 (1962), 603-632.

\smallskip

\bibitem{JFAGW65} J. F. Adams and G. Walker, On Complex Stiefel Manifolds, Proc. Cambridge Phil. Soc. 61 (1965), 81-103.

\smallskip

\bibitem{ABS63} M.F. Atiyah, R. Bott, and A. Shapiro, Clifford Modules, Topology 3 (1963), 3-38.

\smallskip

\bibitem{MFAJAT60} M.F. Atiyah and J.A. Todd, On Complex Stiefel Manifolds, Proc. Cambridge Phil. Soc. 56 (1960), 342-353.

\smallskip

\bibitem{IBJM02} I. B\'ar\'any and J. Matou\v sek, Equipartition of Two Measures by a 4-fan, Discrete Comp. Geom. 27 (2002), 293-301.

\smallskip

\bibitem{IBJM01}  I. B\'ar\'any and J. Matou\v sek, Simultaneous Partitions of Measures by k-fans, Discrete Comp. Geom. 25 (2001), 319-334.

\smallskip

\bibitem{JCB72} J.C. Becker, The Span of Spherical Space Forms, Amer. Jour. Math. 94 (1972), no. 4, 991-1026.

\smallskip

\bibitem{WBAZ} W.A. Beyer and A. Zardecki, The Early History of the Ham Sandwich Theorem, Amer. Math. Monthly 11 (2004), no. 1, 58-61.

\smallskip

\bibitem{ } P. Blagojevi\'c, F. Frick, A. Hasse, and G. Ziegler. Hyperplane Mass Partitions via Relative Equivariant Obstruction Theory, \textit{Documenta Mathematica}, Vol. 21 (2016), 735-771

\smallskip

\bibitem{} P. Blagojevi\'c, F. Frick, A. Hasse, and G. Ziegler. Topology of the Gr\"unbaum-Hadwiger-Ramos Hyperplane Mass Partition Problem, arXiv:15022.02975 [math.AT] 

\smallskip

\bibitem{ } P. Blagojevi\'c and R. Karasev, Extensions of Theorems of Rattray and Makeev, \textit{Top. Methods in Nonlinear Anal.}, Vol. 40, No. 1 (2012) 189--213.

\smallskip 

\bibitem{Gr60} B. Gr\"unbaum, Partitions of Mass-Distributions and Convex Bodies by Hyperplanes, Pacific J. Math. 10 (1960), 1257-1261.

\smallskip

\bibitem{H02} A. Hatcher, Algebraic Topology, Cambridge University Press, 2002.

\smallskip

\bibitem{ }  R. Karasev. Equipartition of a Measure by $Z_p^k$-Invariant Fans, \textit{Discrete Comput. Geom.}, Vol. 43 (2010) 477--481. 

\smallskip

\bibitem{MPS80} N. Mahammed, R. Piccinini, and U. Suter, Some Applications of Topological K-Theory, North-Holland Publishing Co., Mathematics Studies 45, 1980.

\smallskip

\bibitem{Ma07} V.V. Makeev, Equipartition of a Continuous Mass Distribution, Jour. Math. Sci. 140 (2007), no. 4, 551-557. 

\smallskip 

\bibitem{MVZ06} P. Mani-Levitska, S. Vre\'cica, R. \v Zivaljevi\'c, Topology and Combinatorics of Partitions of Masses by Hyperplanes, Adv. Math. 207 (2006), no. 1, 266-296. 

\smallskip

\bibitem{Ma03} J. Matou\v sek, Using the Borsuk-Ulam Theorem: Lectures on Topological Methods in Combinatorics and Geometry, Springer-Verlag, 2003.

\smallskip

\bibitem{MS74} J.W. Milnor and J.D. Stasheff, Characteristic Classes, Princeton University Press, Annals of Mathematics Studies 76, 1974.

\smallskip

\bibitem{MO06} M. Obiedat, Real, Complex, and Quaternionic Equivariant Vector Fields on Spheres, Topology Appl. 53 (2006), 2182-2189.

\smallskip

\bibitem{TO01} T. \"Onder, Equivariant Cross Sections of Complex Stiefel Manifolds, Topology Appl. 109 (2001), 107-125. 

\smallskip

\bibitem{EAR96} E. A. Ramos, Equipartition of Mass Distributions by Hyperplanes, Discrete Comp. Geom. 15 (1996), 147-167.

\smallskip

\bibitem{} S. Simon. $G$-Ham Sandwich Theorems: Balancing Measures by Finite Subgroups of Spheres, \textit{J. Comb. Theory Ser. A}, Vol. 120 (2013) 1906--1912.

\smallskip


\bibitem{RHS64} R. H. Szczarba, On Tangent Bundles of Fibre Spaces and Quotient Spaces, Amer. Jour. Math. 86 (1964), 685-697.

\smallskip

\bibitem{ } S.T. Vre\'cica and R.T. \v Zivaljevi\'c. Conical Equipartitions of Mass Distributions, \textit{Discrete Comput. Geom.}, Vol. 25 (2001) 335--350.

\smallskip

\bibitem{VZ92} S.T. Vre\'cica and R.T. \v Zivaljevi\'c, The Ham Sandwich Theorem Revisited, Israel J. Math. 78 (1992), 21-32.

\smallskip

\bibitem{Zi96} R. \v Zivaljevi\'c, User's Guide to Equivariant Methods in Combinatorics, Publ. Inst. Math. Belgrade, 59 (1996), no. 73, 114-130.

\smallskip

\bibitem{Zi98} R. \v Zivaljevi\'c, User's Guide to Equivariant Methods in Combinatorics II, Publ. Inst. Math. Belgrade, 64 (1998), no. 78, 107-132.

\end{thebibliography}

\end{document}